\documentclass[12pt,letterpaper]{amsart}
\usepackage[utf8]{inputenc}
\usepackage{amsmath}
\usepackage{amssymb}
\usepackage{amsthm}
\usepackage{dsfont}
\usepackage{amsfonts}
\usepackage{mathtools}
\usepackage{relsize,exscale}
\usepackage{indentfirst}
\usepackage{graphicx}
\usepackage{thm-restate}
\usepackage{enumitem}
\usepackage[hidelinks]{hyperref}
\hypersetup{
colorlinks=true,
linkcolor=blue
}

\usepackage[pagewise]{lineno}

\title[Maximal representations in lattices]{Maximal representations in lattices of the symplectic group}
\author{Jacques Audibert}
\thanks{Affiliation: Max Planck Institute for Mathematics, Bonn}
\address{
Max Planck Institute for Mathematics, Vivatsgasse 7, 53111 Bonn, Germany.}
\email{audibert.j@outlook.fr}

\theoremstyle{plain}

\newtheorem{proposition}{Proposition}[subsection]
\newtheorem{theorem}[proposition]{Theorem}

\newtheorem{lemma}[proposition]{Lemma}
\newtheorem{thmx}{Theorem}[section]
\newtheorem{propx}[thmx]{Proposition}

\theoremstyle{definition}

\newcommand{\SL}{\textnormal{SL}}
\DeclareMathOperator{\SU}{SU}
\DeclareMathOperator{\GL}{GL}

\DeclareMathOperator{\Qbar}{\overline{\mathds{Q}}}

\DeclareMathOperator{\Gal}{Gal}
\DeclareMathOperator{\Aut}{Aut}

\DeclareMathOperator{\PSL}{PSL}
\DeclareMathOperator{\PGL}{PGL}
\DeclareMathOperator{\Br}{Br}
\DeclareMathOperator{\Det}{Det}

\DeclareMathOperator{\PSp}{PSp}

\DeclareMathOperator{\PO}{PO}
\DeclareMathOperator{\Diag}{Diag}

\DeclareMathOperator{\Nrd}{Nrd}
\newcommand{\Sp}{\textnormal{Sp}}
\DeclareMathOperator{\Imm}{Im}

\DeclareMathOperator{\MCG}{MCG}
\DeclareMathOperator{\HH}{H}
\DeclareMathOperator{\OO}{O}

\makeatletter
\def\@tocline#1#2#3#4#5#6#7{\relax\begin{comment}
  \ifnum #1>\c@tocdepth 
  \else
    \par \addpenalty\@secpenalty\addvspace{#2}%
    \begingroup \hyphenpenalty\@M
    \@ifempty{#4}{%
      \@tempdima\csname r@tocindent\number#1\endcsname\relax
    }{%
      \@tempdima#4\relax
    }%
    \parindent\z@ \leftskip#3\relax \advance\leftskip\@tempdima\relax
    \rightskip\@pnumwidth plus4em \parfillskip-\@pnumwidth
    #5\leavevmode\hskip-\@tempdima
      \ifcase #1
       \or\or \hskip 1em \or \hskip 2em \else \hskip 3em \fi%
      #6\nobreak\relax
    \dotfill\hbox to\@pnumwidth{\@tocpagenum{#7}}\par
    \nobreak
    \endgroup
  \fi}
\makeatother

\begin{document}

\maketitle

\begin{abstract}
    We prove that all lattices of $\Sp(2n,\mathds{R})$, except those commensurable with $\Sp(4k+2,\mathds{Z})$ when $n=2k+1$, contain the image of infinitely many mapping class group orbits of Zariski-dense maximal representation that are continuous deformations of \emph{maximal diagonal representations}. In particular, we show that $\Sp(4k,\mathds{Z})$ contain Zariski-dense surface subgroups for all $k\geq 1$.
\end{abstract}

\section{Introduction}

Let $\Lambda$ be a lattice of a Lie group $G$. A subgroup $\Gamma<\Lambda$ is \emph{thin} if it is of infinite index in $\Lambda$ yet Zariski-dense in $G$. Thin subgroups satisfy the Superstrong Approximation Theorem \cite{Golsefidy_Expansioninperfectgroups} and are thus an active field of research, see \cite{Sarnak_NotesThinMatrixGroups} and \cite{Kontorovich_Whatisathingroup} for an introduction. In this paper we construct thin subgroups in lattices of $\Sp(2n,\mathds{R})$, $n\geq2$, isomorphic to the fundamental group of closed orientable surfaces of genus at least $2$, called surface subgroups.

In particular we address the question of whether $\Sp(2n,\mathds{Z})$ contains a Zariski-dense surface subgroup. In this regard, we prove the following. \begin{thmx}
\label{theoremSP4k}
    For any $k\geq1$, $\Sp(4k,\mathds{Z})$  contains Zariski-dense surface subgroups.
\end{thmx}
\noindent
These are necessarily thin \cite{Aramayona_GeometricDimensionLattices}. Note that the case of $\Sp(4,\mathds{Z})$ was treated in Long-Thistlethwaite \cite{Long_ZariskidensesurfaceSL4Z} by a different method.

There has been various works on constructing thin surface subgroups in lattices. In their celebrated theorem, Kahn-Markovic \cite{Kahn_Immersingsurfacesinhyerbolicthreemanifold} exhibited surface subgroups in all uniform lattices of $\textrm{SO}(3,1)$. Their technique was dynamical and has been generalized to other Lie groups by Hamenstädt \cite{Hamenstadt_Incompressiblesurfaceslocallysymmetricspaces} and Kahn-Labourie-Mozes \cite{Kahn_Surfacegroupsinuniformlattices}. However these constructions do not apply to $\Sp(2n,\mathds{R})$ \cite[\S 1.2]{Kahn_Surfacegroupsinuniformlattices}. On the other hand, there have been several investigations using \emph{Hitchin representations}, see for instance \cite{Long_ZariskidensesurfaceSL2k+1Z}. In the author's previous work \cite{Audibert_Zariskidensesurfacegroups} and \cite{Audibert_ZariskidenseHitchinrepresentationsuniformlattices}, we constructed thin surface subgroups as images of Hitchin representations in all lattices of the symplectic group except $\Sp(2n,\mathds{Z})$ .

In this paper, we construct \emph{maximal representations} in lattices that are not Hitchin. Maximal representations are defined by the maximality of the Toledo invariant, see \cite{Burger_SurfacegrouperepresentationsmaximalToledo}. They form connected components of $\textrm{Hom}(\pi_1(S_g),\Sp(2n,\mathds{R}))$, for $S_g$ a connected closed orientable surface of genus $g\geq2$, and have been shown to be faithful and discrete by Burger-Iozzi-Wienhard \cite{Burger_SurfacegrouperepresentationsmaximalToledo}. Since then, they have been extensively studied and are now a prototypical example of a higher Teichmüller space \cite{Wienhard_InvitationHigherTeichmuller}.

Let $n\geq2$. Denote by $\phi_n:\SL(2,\mathds{C})\to\Sp(2n,\mathds{C})$ the morphism\footnote{We denote also by $\phi_n$ the induced embedding of $\PSL(2,\mathds{C})$ in $\PSp(2n,\mathds{C})$.}
\begin{align*}
A&\mapsto\begin{pmatrix}
A&&\\
&\dots&\\
&&A
\end{pmatrix}
\end{align*}
which we call \emph{the diagonal embedding}. A representation of the form $\phi_n\circ j$, with $j:\pi_1(S)\to\SL(2,\mathds{R})$ a discrete and faithful representation, is a maximal representation \cite{Guichard_TopologicalinvariantsofAnosovrep}. We call representations of this form \emph{maximal diagonal representations}. Any continuous deformation of a maximal diagonal representation is maximal, and hence faithful. This will play a major role in our construction, that we now describe.

We begin by classifying lattices of $\Sp(2n,\mathds{R})$ that contain the image of a maximal diagonal representation.
The classification is done using non-abelian Galois cohomology, the cohomology theory that classifies forms of algebraic groups and thus arithmetic lattices of Lie groups. The main difficulty here is to deal with the ``twist" of cocycles by the centralizer of $\Imm(\phi_n)$. Because of this, a dichotomy appears between $n$ even and $n$ odd.

\begin{propx}
\label{propositionclassificationdiagonal}
    Every lattice of $\Sp(2n,\mathds{R})$, not widely commensurable\footnote{We say that two subgroup of $\Sp(2n,\mathds{R})$ are \emph{widely commensurable} if, up to conjugation, their intersection has finite index in both of them.} with $\Sp(4k+2,\mathds{Z})$ when $n=2k+1$, contains the image of a maximal diagonal representation of some genus. Conversely, when $n=2k+1$, $\Sp(4k+2,\mathds{Z})$ does not contain the image of a maximal diagonal representations of any genus.
\end{propx}

See Proposition \ref{propclassificationlatticephi} for a more precise version of this result, which strengthens the distinction between the even and the odd case. Given a lattice $\Lambda$ that contains the image of a maximal diagonal representation of some surface group, we deform the latter to have Zariski-dense image but still lie in $\Lambda$. This is done by a ``bending" technique, as introduced by Johnson-Millson \cite{Johnson_DeformationSpacesCompactHyperbolicManifolds}. This requires to find a simple closed curve on the surface and an element of $\Lambda$ that commutes with its image. If the curve is separating, the bending consists of conjugating part of the representation by the element of $\Lambda$. If it is not separating, it can be described in terms of HNN-extension. Using Galois cohomology, we are able to compute the centralizer of the simple closed curve in $\Lambda$. Using the classification of Zariski-closures of maximal representations by Hamlet-Pozzetti \cite{Hamlet_Classificationtighthomomorphisms}, we show that there are ways to bend the maximal diagonal representation that make it Zariski-dense. This yields the following theorem, of which Theorem \ref{theoremSP4k} is a consequence.

\begin{thmx}
\label{theoremdiagonal}
    Let $n\geq2$ and $\Lambda$ be a lattice of $\Sp(2n,\mathds{R})$ not widely commensurable with $\Sp(4k+2,\mathds{Z})$ when $n=2k+1$. Then there exists $g\geq2$ such that $\Lambda$ contains infinitely many $\MCG(S_g)$-orbits of Zariski-dense maximal representations of $\pi_1(S_g)$. Furthermore these representations are continuous deformations of maximal diagonal representations.
\end{thmx}

Here $\MCG(S_g)$ denote the mapping class group of $S_g$. It acts on conjugacy classes of representations by precomposition. We distinguish different $\MCG(S_g)$-orbits of representations in $\Lambda$ by the set of primes for which its reduction surjects. We know that, for Zariski-dense representations, this set is finite by the Strong Approximation Theorem \cite{Matthews_CongruencepropertiesofZariskidensesubgroups}. In Lemma \ref{lemmareduction}, we show that it is an invariant of mapping class group orbits.
\\

\noindent
\textbf{Organisation of the paper.}
In Section \ref{sectionBackground} we provide the necessary background on quaternion algebras and non-abelian Galois cohomology. In Section \ref{sectionthediagonalembedding}, we classify lattices of the symplectic group that contain the image of a maximal diagonal representation. Finally, in Section \ref{sectiondeformationsofmaximaldiagonalrepresentations}, we deform maximal diagonal representations into Zariski-dense ones and prove that this construction produces infinitely many mapping class group orbits of representations.
\\

\section{Background}
\label{sectionBackground}

\subsection{Quaternion algebras}
\label{subsectionformsofthesymplecticgroup}

Let $k$ be a field of characteristic different from $2$. A \emph{quaternion algebra} over $k$ is a $4$-dimensional associative $k$-algebra which admits a basis $\{1,i,j,ij\}$ such that $1$ is the identity element and \begin{equation*}
        i^2=a1,\ j^2=b1\ \textrm{and}\ ij=-ji
    \end{equation*} with $a,b\in k^{\times}$. We denote such a quaternion algebra by $(a,b)_k$.
Note that a quaternion algebra is necessarily non commutative. The subspace spanned by $1$ can be identified with $k$.

Conversely, for any choice of $a$ and $b$ non zero, there is a quaternion algebra $(a,b)_k$ defined by the relations above. For instance, the quaternion algebra $(1,1)_k$ is isomorphic to $\textrm{M}_2(k)$. Over a number field, there are infinitely many non-isomorphic quaternion algebras \cite[Theorem 7.3.6]{Maclachlan_ArithmeticHyperbolic3Manifolds}.

Every quaternion algebra $A$ comes with a \emph{conjugation}, define as
\begin{equation*}
    x=x_0+x_1i+x_2j+x_3ij\mapsto\overline{x}=x_0-x_1i-x_2j-x_3ij.
\end{equation*}
It allows one to define the \emph{reduced norm} as $\Nrd:A\to k,\ x\mapsto x\overline{x}$.

Let $F$ be a number field and $\mathcal{O}_F$ its ring of integers. Let $A$ be a quaternion algebra over $F$. An \emph{order} of $A$ is a finitely generated $\mathcal{O}_F$-submodule of $A$ containing 1 which generates $A$ as a vector space and which is a subring of $A$.
For instance, if $a,b\in\mathcal{O}_F$ are non-zero, then $\mathcal{O}_F[i,j,ij]$ is an order of $(a,b)_F$. Note that orders are the anologues of rings of integers for number fields.

Suppose that $F$ is a totally real number field. Given an embedding $\sigma:F\hookrightarrow\mathds{R}$, we denote by $\mathds{R}^{\sigma}$ the field $\mathds{R}$ with the $F$-alebgra structure coming from $\sigma$. We say that $A$ \emph{splits} over $\sigma$ if $A\otimes_F\mathds{R}^{\sigma}\simeq\textrm{M}_2(\mathds{R})$. Suppose that $A$ splits exactly at only at the real place $\sigma$. Then, given an order $\mathcal{O}$ of $A$,
\begin{equation*}
    \mathcal{O}^1:=\{x\in\mathcal{O}\mid\Nrd(x)=1\}
\end{equation*}
is a lattice in $\{x\in A\otimes_F\mathds{R}^{\sigma}\mid\Nrd(x)=1\}\simeq\SL(2,\mathds{R})$.
It is non-uniform if and only if $F=\mathds{Q}$ and $A\simeq\textrm{M}_2(\mathds{Q})$.
Denote by $\SL(n,\mathcal{O})$ the set of $n$-by-$n$ matrices with coefficients in $\mathcal{O}$ that have determinant $1$ when viewed in $\textrm{M}_n(A\otimes_F\mathds{R}^{\sigma})\simeq\textrm{M}_{2n}(\mathds{R})$.
Let us define
\begin{equation*}
\SU(I_n,\overline{\phantom{s}};\mathcal{O}):=\{M\in\SL(n,\mathcal{O})\mid\overline{M}^t M=I_n\}
\end{equation*}
where $M^t$ is the transpose of $M$ viewed as an $n$-by-$n$ matrix.
The subgroup $\SU(I_n,\overline{\phantom{s}};\mathcal{O})$ is a lattice of $\Sp(2n,\mathds{R})$ which is non-uniform if and only if $F=\mathds{Q}$ \cite{Morris_IntroductionArithmeticGroups}. Conversely, all lattices of $\Sp(2n,\mathds{R})$ are widely commensurable to one of those.

\begin{proposition}
\label{propclassifactionlattice}
Every lattice in $\Sp(2n,\mathds{R})$ is widely commensurable with $\SU(I_n,\overline{\phantom{s}};\mathcal{O})$ for $\mathcal{O}$ an order in a quaternion algebra $A$ over a totally real number field $F$ such that $A$ splits over exactly one real place of $F$.
\end{proposition}

See \cite[Proposition 1.5]{Audibert_Zariskidensesurfacegroups} and \cite[Proposition 1.10]{Audibert_ZariskidenseHitchinrepresentationsuniformlattices} for a proof. Note that, given two distinct orders $\mathcal{O}$ and $\mathcal{O}'$ of the same quaternion algebra, the lattices $\SU(I_n,\overline{\phantom{s}},\mathcal{O})$ and $\SU(I_n,\overline{\phantom{s}},\mathcal{O}')$ are commensurable.

\subsection{Galois cohomology sets}
\label{subsectionGaloiscohomology}

The main reference for this section is \S1.3.2 of Platonov-Rapinchuck \cite{Platonov_AlgebraicgroupsNumbertheory}.
Let $A$ be a topological group acting continuously on a discrete group $M$ which is not necessarily abelian. A \emph{1-cocycle} is a continuous map $f:A\to M$ that satisfies
$$f(st)=f(s)s(f(t))$$ for all $s,t\in A$.
We say that two 1-cocycles $f$ and $f'$ are \emph{equivalent} if there exists an element $m\in M$ such that $$f'(s)=m^{-1}f(s)s(m)$$ for all $s\in A$. We denote by $\HH^1(A,M)$ the set of equivalence classes of $1$-cocycles and call it \emph{the first cohomology set of $A$ (with coefficients in $M$)}. It is not necessarily a group but a set with distinguished element. However, when $M$ is abelian, $\HH^1(A,M)$ is the usual first continuous cohomology group.

The setting we will be interested in is the following. Let $F$ be a number field. Consider $\Gal(\Qbar/F)$ with the profinite topology acting by conjugation on the automorphism group\footnote{We restrict to algebraic automorphisms.} $\Aut(\Sp(2n,\Qbar))\simeq\PSp(2n,\Qbar)$ endowed with the discrete topology. In this case, the set $\HH^1(\Gal(\Qbar/F),\PSp(2n,\Qbar))$ classifies \emph{$\Qbar/F$-forms of $\Sp_{2n}$}, i.e$.$ algebraic groups $\mathds{G}$ defined over $F$ which are isomorphic to $\Sp_{2n}$ over $\Qbar$.
Indeed, given a $1$-cocycle $\zeta:\Gal(\Qbar/F)\to\PSp(2n,\Qbar)$, one can associate a $\Qbar/F$-form of $\Sp_{2n}$ as
    \begin{equation*}
        \prescript{}{\zeta}{\mathds{G}}(R):=\{g\in\mathds{G}(\Qbar\otimes_FR)\mid \zeta(\sigma)(\sigma(g))\zeta(\sigma)^{-1}=g\ \forall\sigma\in\Gal(\Qbar/F)\}
    \end{equation*}
for any $F$-algebra $R$.

\begin{proposition}[Proposition 1 Chapter III \cite{Serre_GaloisCohomology}]
    The map that associates to a $1$-cocycle a $\Qbar/F$-form of $\Sp_{2n}$
    defines a bijection between $\HH^1(\Gal(\Qbar/F),\PSp(2n,\Qbar))$ and the set of isomorphism classes of $\Qbar/F$-forms of $\Sp_{2n}$.
\end{proposition}

\subsection{Exact sequences}

Let $Z(M)$ be the center of $M$ and $N<Z(M)$. Then there is a long exact sequence of sets with distinguished elements
\begin{equation*}
\HH^1(A,N)\to\HH^1(A,M)\to\HH^1(A,M/N)\xrightarrow{\delta}\HH^2(A,N),
\end{equation*}
where $\HH^2(A,N)$ is usual group cohomology since $N$ is abelian. The morphism $\delta$ is called the \emph{connecting morphism} and is constructed in the following way. Let $f:A\to M/N$ a $1$-cocycle. For every $a\in A$ pick $\tilde{f}(a)\in M$ which projects to $f(a)$ in $M/N$. For every $a,b\in A$, we define $$\delta(f)_{a,b}=\tilde{f}(a)a(\tilde{f}(b))\tilde{f}(ab)^{-1}\in N.$$ It is a factor set whose class in $\HH^2(A,N)$ only depends on the class of $f$.

\subsection{The Brauer group}

The main reference for this section is \S 2.4 in Gille-Szamuely \cite{Gille_CentralSimpleAlgebrasGaloisCohomology}.
Let $k$ be a field. A \emph{central simple algebra} $A$ is a unital finite-dimensional associative $k$-algebra whose center is $k$ and which has no two-sided ideal besides $A$ and $\{0\}$. Quaternion algebras are exactly the central simple algebras of dimension 4. We say that a central simple algebra is a \emph{division algebra} if every non-zero element is invertible.

Wedderburn's Theorem \cite[Theorem 2.9.6]{Maclachlan_ArithmeticHyperbolic3Manifolds} asserts that any central simple algebra is isomorphic to $\textrm{M}_n(D)$, where $D$ is a division algebra over $k$, which is unique up to isomorphism. Let $A_1=\textrm{M}_{n_1}(D_1)$ and $A_2=\textrm{M}_{n_2}(D_2)$ be two central simple algebras and $D_1$ and $D_2$ two division algebras. We say that $A_1$ and $A_2$ are \emph{Brauer equivalent} if $D_1$ is isomorphic to $D_2$. 

The tensor product of two central simple algebras is a central simple algebra. Hence the tensor product induces a group law on Brauer equivalence classes of central simple algebras with identity element $[k]$ and where the inverse of $[A]$ is the class of the opposite algebra of $A$. This group is abelian and is called the \emph{Brauer group of $k$}, denoted by $\Br(k)$.
It has an interpretation in terms of group cohomology. Denote by $\overline{k}$ the separable closure of $k$, which is the algebraic closure of $k$ when the characteristic of $k$ is $0$.
\begin{proposition}[Theorem 4.4.7 in \cite{Gille_CentralSimpleAlgebrasGaloisCohomology}]
    The Brauer group of $k$ is isomorphic to $\HH^2(\Gal(\overline{k}/k),\overline{k}^{\times})$.
\end{proposition}

Let $F$ be a number field. Denote by $\mu_2$ the subgroup $\{\pm1\}$ of $\Qbar^{\times}$. Then $\HH^2(\Gal(\Qbar/F),\mu_2)$ is a subgroup of $\HH^2(\Gal(\Qbar/F),\Qbar^{\times})$.

\begin{proposition}[Theorem 20 Chapter X \cite{Albert_Structureofalgebras}]
\label{Propmu2brauergroup}
Under the isomorphism $\Br(F)\simeq \HH^2(\Gal(\Qbar/F),\Qbar^{\times})$, $\HH^2(\Gal(\Qbar/F),\mu_2)$ corresponds to the subgroup of $\Br(F)$ consisting of equivalence classes of quaternion algebras.
\end{proposition}

\section{The diagonal embedding}
\label{sectionthediagonalembedding}

\subsection{Compatible cocycles}

Let us introduce some notations. Fix $n\geq2$ and define
\begin{equation*}
K=\begin{pmatrix}
    0&1\\
    -1&0
\end{pmatrix}.
\end{equation*} Recall that $\phi_n:\SL(2,\mathds{C})\to\Sp(2n,\mathds{C})$ is the diagonal embedding and let $K_n=\phi_n(K)$. The latter induces a symplectic form on $\mathds{C}^{2n}$ that is preserved by the image of $\phi_n$. From now on, the symplectic groups we will consider will always be the ones of the form $K_n$.

The \emph{Kronecker product} of two matrices $A=(a_{ij})\in\textrm{M}_n(\mathds{C})$ and $B\in\textrm{M}_m(\mathds{C})$ is the matrix $A\otimes B\in\textrm{M}_{mn}(\mathds{C})$ defined by
\begin{equation*}
    A\otimes B=\begin{pmatrix}
        a_{11}B&\dots&a_{1n}B\\
        \dots&\dots&\dots\\
        a_{n1}B&\dots&a_{nn}B
    \end{pmatrix}.
\end{equation*}
It satisfies $(A\otimes B)(C\otimes D)=AC\otimes BD$ and $\Det(A\otimes B)=\Det(A)^m\Det(B)^n$. Note that $\phi_n(A)=I_n\otimes A$. If $M\in\OO(I_n,\mathds{C})$ and $A\in\SL(2,\mathds{C})$, then $M\otimes A\in\Sp(2n,\mathds{C})$. 

\begin{lemma}
\label{lemmacentralizerImphin}
The centralizer of $\Imm(\phi_n)$ in $\Sp(2n,\mathds{C})$ is $\OO(I_n,\mathds{C})\otimes I_2$.    
\end{lemma}

\begin{proof}
It is clear that $\OO(I_n,\mathds{C})\otimes I_2$ commutes with $\Imm(\phi_n)$. Conversely, if a matrix $X$ commutes with $\Imm(\phi_n)$ then each of its $2$-by-$2$ blocks have to be scalar multiple of $I_2$, i.e$.$ $X$ has to be of the form $M\otimes I_2$ for some $n$-by-$n$ matrix $M$. Then $X\in\Sp(2n,\mathds{C})$ if and only if $M\in\OO(I_n,\mathds{C})$.
\end{proof}

Let $F$ be a totally real number field and $\xi:\Gal(\Qbar/F)\to\PSL(2,\Qbar)$ a $1$-cocycle. We say that a $1$-cocycle $\zeta:\Gal(\Qbar/F)\to\PSp(2n,\Qbar)$ is \emph{$\phi_n$-compatible} with $\xi$ if
\begin{equation*}
\phi_n(\prescript{}{\xi}{\SL_2}(F))<\prescript{}{\zeta}{\Sp_{2n}}(F).
\end{equation*} For instance $I_n\otimes\xi :\Gal(\Qbar/F)\to\PSp(2n,\Qbar)$ is a $1$-cocycle $\phi_n$-compatible with $\xi$.

\begin{lemma}
\label{lemmaphicompatiblecocycle}
A $1$-cocycle $\zeta:\Gal(\Qbar/F)\to\PSp(2n,\Qbar)$ is $\phi_n$-compatible with $\xi$ if and only if $\zeta=\eta\otimes\xi$ where $\eta:\Gal(\Qbar/F)\to\PO(I_n,\Qbar)$ is a $1$-cocycle.
\end{lemma}

\begin{proof}
Let $\zeta:\Gal(\Qbar/F)\to\PSp(2n,\Qbar)$ be a $1$-cocycle $\phi_n$-compatible with $\xi$. Fix $\sigma\in\Gal(\Qbar/F)$. For all $A\in\prescript{}{\xi}{\SL_2}(F)$
\begin{equation*}
\zeta(\sigma)\sigma(\phi_n(A))\zeta(\sigma)^{-1}=\phi_n(\xi(\sigma)\sigma(A)\xi(\sigma)^{-1}),
\end{equation*} which implies that $\phi_n(\xi(\sigma))^{-1}\zeta(\sigma)$ commutes with $\phi_n(\prescript{}{\xi}{\SL_2}(F))$. The latter is Zariski-dense in $\Imm(\phi_n)$. By Lemma \ref{lemmacentralizerImphin}, for all $\sigma\in\Gal(\Qbar/F)$ there exists $\eta(\sigma)\in\PO(I_n,\Qbar)$ such that $$\zeta(\sigma)=\phi_n(\xi(\sigma))(\eta(\sigma)\otimes I_2 )=\eta(\sigma)\otimes\xi(\sigma).$$ Since $\zeta$ is a $1$-cocycle, $\eta:\Gal(\Qbar/F)\to\PO(I_n,\Qbar)$ is also a $1$-cocycle. Conversely for any such $1$-cocycle $\eta$, $\eta\otimes\xi$ is $\phi_n$-compatible with $\xi$.
\end{proof}

\subsection{Forms of $\textrm{Sp}(2n,\mathds{R})$}

The goal of this section is to compute the forms associated to compatible cocycles. We prove here Proposition \ref{propclassificationlatticephi}.

Consider the short exact sequence of groups
\begin{equation*}
\label{equationpi}
    1\to\mu_2\to\OO(n,\Qbar)\rightarrow\PO(n,\Qbar)\to1.
\end{equation*} It induces a connecting map
\begin{equation*}
\label{equationpartialn}
\HH^1(\Gal(\Qbar/F),\PO(n,\Qbar))\xrightarrow[]{\partial_n}\HH^2(\Gal(\Qbar/F),\mu_2).
\end{equation*}

\begin{lemma}
\label{lemmapartialsurjective}
If $n$ is even, $\partial_{n}$ is surjective. If $n$ is odd, $\partial_{n}$ is trivial.
\end{lemma}

Before proving the lemma, we introduce some notation. Let $a,b\in F^{\times}$ and define
\begin{align*}
    T^{a,b}:&\Gal(\Qbar/F)\to\PSL(2,\Qbar)\\
    &\sigma\mapsto\left\{\begin{array}{ll}
        I_2 & \mbox{if $\sigma(\sqrt{a})=\sqrt{a}$ and $\sigma(\sqrt{b})=\sqrt{b}$} \\
        \begin{pmatrix}
            1&0\\
            0&$-$1
        \end{pmatrix} & \mbox{if $\sigma(\sqrt{a})=\sqrt{a}$ and $\sigma(\sqrt{b})=-\sqrt{b}$}\\
        \begin{pmatrix}
            0&1\\
            1&0
        \end{pmatrix} & \mbox{if $\sigma(\sqrt{a})=-\sqrt{a}$ and $\sigma(\sqrt{b})=\sqrt{b}$}\\
        \begin{pmatrix}
            0&1\\
            $-$1&0
        \end{pmatrix} & \mbox{if $\sigma(\sqrt{a})=-\sqrt{a}$ and $\sigma(\sqrt{b})=-\sqrt{b}$.}
    \end{array}\right.
\end{align*}
We often write $T_{\sigma}$ or $T_{\sigma}^{a,b}$ instead of $T^{a,b}(\sigma)$. It is a $1$-cocycle. Explicit computations show that $$\prescript{}{T^{a,b}}{\textrm{M}_2}(\Qbar):=\{M\in\textrm{M}_2(\Qbar)\mid T^{a,b}_{\sigma}\sigma(M) (T^{a,b}_{\sigma})^{-1}=M,\ \forall\sigma\in\Gal(\Qbar/F)\}$$ is the quaternion algebra
\begin{equation*}
    \left\{\begin{pmatrix}
x_0+\sqrt{a}x_1&\sqrt{b}x_2+\sqrt{ab}x_3\\
        \sqrt{b}x_2-\sqrt{ab}x_3&x_0-\sqrt{a}x_1
    \end{pmatrix}\ \big|\ x_i\in F\right\}\simeq(a,b)_F.
\end{equation*}

\begin{proof}[Proof of Lemma \ref{lemmapartialsurjective}]
Suppose that $n$ is even. Let $a,b\in F^{\times}$ and $$\chi_n:=\phi_{\frac{n}{2}}(T^{a,b}):\Gal(\Qbar/F)\to\PO(n,\Qbar).$$ For all $\sigma\in\Gal(\Qbar/F)$ pick a lift $M_{\sigma}\in\OO(n,\Qbar)$ of $\chi_n(\sigma)$. By construction of the connecting map, for all $s,t\in\Gal(\Qbar/F)$, $\partial_n(\chi_n)_{st}I_n=M_ss(M_t)M_{st}^{-1}$. Then for all $s,t$
$$\partial_n(\chi_n)_{st}=\partial_2(\chi_2)_{st}.$$
 Explicit computations show that $\partial_2(\chi_2)$ corresponds to $(a,b)_F$ under the embedding $\HH^2(\Gal(\Qbar/F),\mu_2)\hookrightarrow\Br(F)$. Since this is true for every $a,b\in F^{\times}$, $\partial_2$ and hence $\partial_n$ is surjective.

Suppose that $n$ is odd. Let $\eta:\Gal(\Qbar/F)\to\PO(n,\Qbar)$ be a $1$-cocycle. For all $\sigma\in\Gal(\Qbar/F)$ pick a lift $M_{\sigma}\in\OO(n,\Qbar)$ of $\eta(\sigma)$. By construction of $\partial_n$, for all $s,t\in\Gal(\Qbar/F)$ $$\partial_n(\eta)_{st}I_n=M_ss(M_t)M_{st}^{-1}\in\mu_2.$$ Define $f:\Gal(\Qbar/F)\to\Qbar$, $\sigma\mapsto\Det(M_\sigma)$. Then for all $s,t\in\Gal(\Qbar/F)$ $$\partial_n(\eta)_{st}=f(s)s(f(t))f(st)^{-1}$$ which implies that $\partial_n(\eta)$ is trivial.
\end{proof}

Consider the short exact sequence of groups
\begin{equation*}
1\to\mu_2\to\Sp(2n,\Qbar)\to\PSp(2n,\Qbar)\to1
\end{equation*}where $\mu_2$ is the group $\{\pm I_{2n}\}$. It induces a long exact sequence of sets with distinguished elements
\begin{equation*}
1\to\HH^1(\Gal(\Qbar/F),\PSp(2n,\Qbar))\xrightarrow[]{\delta_{2n}}\HH^2(\Gal(\Qbar/F),\mu_2),
\end{equation*} see \S1.3.2 and Proposition 2.7 in \cite{Platonov_AlgebraicgroupsNumbertheory}. By Theorem 6.20 in \cite{Platonov_AlgebraicgroupsNumbertheory}, $\delta_{2n}$ is surjective.

\begin{lemma}
\label{lemmaproductfactorset}
Let $\eta:\Gal(\Qbar/F)\to\PO(n,\Qbar)$ and $\xi:\Gal(\Qbar/F)\to\PSL(2,\Qbar)$ be two $1$-cocycles. Then
\begin{equation*}
\delta_{2n}(\eta\otimes\xi)=\delta_2(\xi)\partial_n(\eta).
\end{equation*}
\end{lemma}

\begin{proof}
    For each $\sigma\in\Gal(\Qbar/F)$ pick $M_{\sigma}\in\OO(n,\Qbar)$ and $A_{\sigma}\in\SL(2,\Qbar)$ such that the projectivization of $M_{\sigma}$ and $A_{\sigma}$ are $\eta(\sigma)$ and $\xi(\sigma)$ respectively. Then $M_\sigma\otimes A_{\sigma}$ has for projectivization $\eta(\sigma)\otimes\xi(\sigma)$. For all $s,t\in\Gal(\Qbar/F)$
    \begin{align*}
        \delta_{2n}(\eta\otimes\xi)_{st}I_{2n}&=(M_s\otimes A_s)s(M_t\otimes A_t)(M_{st}\otimes A_{st})^{-1}\\
        &=(M_ss(M_t)M_{st}^{-1})\otimes (A_ss(A_t)A_{st}^{-1})\\
        &=\partial_n(\eta)_{st}I_n\otimes\delta_2(\xi)_{st}I_2
    \end{align*}
\end{proof}

\begin{proposition}
\label{propclassificationlatticephi}
Let $\Gamma$ be an arithmetic lattice in $\SL(2,\mathds{R})$ which is commensurable with the elements of norm 1 in an order $\mathcal{O}_{\Gamma}$ of a quaternion algebra over a totally real number field $F$.

If $n$ is even, $\phi_n(\Gamma)$ lies in lattice of $\Sp(2n,\mathds{R})$ widely commensurable with $\SU(I_n,\overline{\phantom{s}};\mathcal{O})$, for any order $\mathcal{O}$ in any quaternion algebra over $F$ that splits at exactly one place of $F$. Furthermore, up to wide commensurability, these are the only lattices of $\Sp(2n,\mathds{R})$ that contain $\phi_n(\Gamma)$.

If $n$ is odd, $\phi_n(\Gamma)$ lies in a lattice of $\Sp(2n,\mathds{R})$ widely commensurable with $\SU(I_n,\overline{\phantom{s}};\mathcal{O}_{\Gamma})$. Furthermore, up to wide commensurability, this is the only lattice of $\Sp(2n,\mathds{R})$ that contains $\phi_n(\Gamma)$.
\end{proposition}

Proposition \ref{propositionclassificationdiagonal} follows from Proposition \ref{propclassificationlatticephi} together with \cite[Proposition 1.6]{Audibert_ZariskidenseHitchinrepresentationsuniformlattices}.

\begin{proof}[Proof of Proposition \ref{propclassificationlatticephi}]
Fix $a,b\in F^{\times}$ such that $\mathcal{O}_{\Gamma}$ is an order of $(a,b)_F$.
The $1$-cocycle $I_n\otimes T^{a,b}:\Gal(\Qbar/F)\to\PSp(2n,\Qbar)$ is $\phi_n$-compatible with $T^{a,b}$. Hence $\phi_n(\prescript{}{T^{a,b}}{\SL_2(F)})<\prescript{}{I_n\otimes T^{a,b}}{\Sp_{2n}(F)}$. By Proposition 5.2 in Appendix A of \cite{Milne_LieAlgebrasAlgebraicGroupsLieGroups}\footnote{Proposition 5.2 is stated for $F=\mathds{Q}$ but works for any number field, as can be seen using restriction of scalars (see \S10.3 in \cite{Maclachlan_ArithmeticHyperbolic3Manifolds}).} $\phi_n(\Gamma)<\prescript{}{I_n\otimes T^{a,b}}{\Sp_{2n}(\mathcal{O}_F)}$ up to finite index. We now describe the group $\prescript{}{I_n\otimes T^{a,b}}{\Sp_{2n}(\mathcal{O}_F)}$. For any $M\in\SL(2n,\Qbar)$
\begin{align*}
    &M\in\prescript{}{I_n\otimes T^{a,b}}{\Sp_{2n}(F)}\\
    \Leftrightarrow&\left\{\begin{array}{ll}
        M\in\textrm{M}_n(\prescript{}{T^{a,b}}{\textrm{M}_2(F)})\ \textrm{and}\\
        M^\top K_nM=K_n
    \end{array}\right.\\
    \Leftrightarrow&\left\{\begin{array}{ll}
        M\in\textrm{M}_n(\prescript{}{T^{a,b}}{\textrm{M}_2(F)})\ \textrm{and}\\
        \overline{M}^tM=I_{2n}
    \end{array}\right.
\end{align*}
since $\overline{M}^t=(K_nMK_n^{-1})^\top$, where $\overline{\phantom{s}}$ is the conjugation on $\prescript{}{T^{a,b}}{\textrm{M}_2(F)}$ and $M^t$ is the transpose of $M$ viewed as an $n$-by-$n$ matrix. Hence $\prescript{}{I_n\otimes T^{a,b}}{\Sp_{2n}}(\mathcal{O}_F)$ is widely commensurable with $\SU(I_n,\overline{\phantom{s}};\mathcal{O}_{\Gamma})$.

Suppose that $n$ is even. Let $(c,d)_F$ be a quaternion algebra over $F$. Let $(a_{st})_{st}\in\HH^2(\Gal(\Qbar/F),\mu_2)$ be a $2$-cocycle corresponding to $(c,d)_F$, see Proposition \ref{Propmu2brauergroup}. Since 
\begin{equation*}
\HH^1(\Gal(\Qbar/F),\PO(n,\Qbar))\xrightarrow[]{\partial_n}\HH^2(\Gal(\Qbar/F),\mu_2)
\end{equation*} is surjective, see Lemma \ref{lemmapartialsurjective}, there exists a $1$-cocycle $\eta:\Gal(\Qbar/F)\to\PO(n,\Qbar)$ such that $\partial_n(\eta)=\delta_2(T^{a,b})^{-1}(a_{st})$. By Lemma \ref{lemmaphicompatiblecocycle}, $\eta\otimes T^{a,b}$ is $\phi_n$-compatible with $T^{a,b}$. Hence Proposition 5.2 in Appendix A of \cite{Milne_LieAlgebrasAlgebraicGroupsLieGroups} shows that $\phi(\Gamma)<\prescript{}{\eta\otimes T^{a,b}}{\Sp_{2n}(\mathcal{O}_F)}$ up to finite index. We now describe the group $\prescript{}{\eta\otimes T^{a,b}}{\Sp_{2n}(\mathcal{O}_F)}$.

Consider the map
\begin{align*}
    \zeta:\Gal(\Qbar/F)&\to\prescript{}{I_n\otimes T^{c,d}}{\PSp_{2n}(\Qbar)}\\
\sigma&\mapsto(\eta(\sigma)\otimes T_{\sigma}^{a,b})(I_n\otimes T^{c,d}_{\sigma})^{-1}.
\end{align*} We endow $\prescript{}{I_n\otimes T^{c,d}}{\PSp_{2n}(\Qbar)}$ with the action of $\Gal(\Qbar/F)$ defined by $\sigma\cdot M:=(I_n\otimes T^{c,d}_{\sigma})\sigma(M)(I_n\otimes T^{c,d}_{\sigma})^{-1}.$
The map $\zeta$ is a $1$-cocycle with respect to this action. By Lemma \ref{lemmaproductfactorset}, $\delta_{2n}(\eta\otimes T^{a,b})=(a_{st})$, so the image of $\zeta$ in $\HH^2(\Gal(\Qbar/F),\mu_2)$ is trivial. Hence it lifts to a $1$-cocycle in $\prescript{}{I_n\otimes T^{c,d}}{\Sp_{2n}(\Qbar)}$. It is thus trivial, see Proposition 2.7 in \cite{Platonov_AlgebraicgroupsNumbertheory}. Hence there exists $S\in\Sp(2n,\Qbar)$ such that for all $\sigma\in\Gal(\Qbar/F)$
$$\eta(\sigma)\otimes T^{a,b}_{\sigma}=S^{-1}(I_n\otimes T^{c,d}_{\sigma})\sigma(S).$$
It follows that $M\in\prescript{}{\eta\otimes T^{a,b}}{\Sp_{2n}(F)}$ if and only if
\begin{align*}
    &\left\{\begin{array}{ll}
       SMS^{-1}\in\textrm{M}_n(\prescript{}{T^{c,d}}{\textrm{M}_2(F)})\  \textrm{and}\\
        (SMS^{-1})^\top S^{-\top}K_nS^{-1}(SMS^{-1})=S^{-\top}K_nS^{-1}
    \end{array}\right.\\
    \Leftrightarrow&\ SMS^{-1}\in\prescript{}{I_n\otimes T^{c,d}}{\Sp_{2n}(F)}
\end{align*} since $S^\top K_nS=K_n$. Finally, $\prescript{}{\eta\otimes T^{a,b}}{\Sp_{2n}(\mathcal{O}_F)}$ is widely commensurable with $\prescript{}{I_n\otimes T^{c,d}}{\Sp_{2n}(\mathcal{O}_F)}$ which is widely commensurable with $\SU(I_n,\overline{\phantom{s}};\mathcal{O})$ for $\mathcal{O}$ an order of $(c,d)_F$.

Conversely, suppose that $\Lambda$ is an arithmetic subgroup of $\Sp(2n,\mathds{R})$ that contains $\phi_n(\Gamma)$. Since $\Sp(2n,\mathds{R})$ is simple, $\Lambda$ is widely commensurable with $\prescript{}{\zeta}{\Sp_{2n}}(\mathcal{O}_L)$ for $L$ a number field and $\zeta:\Gal(\Qbar/L)\to\PSp_{2n}(\overline{\mathds{Q}})$ a $1$-cocycle \cite[Corollary 5.5.16]{Morris_IntroductionArithmeticGroups}. By Proposition 1.6 in \cite{Audibert_ZariskidenseHitchinrepresentationsuniformlattices} we can assume that $L=F$.
We show that $\zeta$ is $\phi_n$-compatible with $T^{a,b}$. For every $\sigma\in\Gal(\Qbar/F)$ denote by
\begin{align*}
    \phi_n^{\sigma}:\SL_2(\Qbar)&\to\Sp(2n,\Qbar)\\ &g\mapsto\zeta(\sigma)(\sigma\circ\phi_n((T^{a,b}_{\sigma})^{-1}\sigma^{-1}(g)T^{a,b}_{\sigma}))\zeta(\sigma)^{-1}.
\end{align*}
This is an algebraic morphism that coincides with $\phi_n$ on a finite-index subgroup of $\Gamma$. Since any finite-index subgroup of $\Gamma$ is Zariski-dense in $\SL_2(\Qbar)$, $\phi_n=\phi_n^{\sigma}$. This means that $\phi_n(\prescript{}{T^{a,b}}{\SL_2}(F))<\prescript{}{\zeta}{\Sp_{2n}}(F)$. Hence it satisfies the assumptions of Lemma \ref{lemmaphicompatiblecocycle}. If $n$ is even, we described above the associated arithmetic group. If $n$ is odd, Lemma \ref{lemmapartialsurjective} shows that there no other arithmetic group than $\SU(I_n,\overline{\phantom{s}},\mathcal{O}_{\Gamma})$ up to wide commensurability.
\end{proof}

\section{Deformations of maximal diagonal representations}
\label{sectiondeformationsofmaximaldiagonalrepresentations}

Let $\Lambda$ be a lattice of $\Sp(2n,\mathds{R})$ that is not widely commensurable with $\Sp(4k+2,\mathds{Z})$ when $n=2k+1$. By Proposition \ref{propclassificationlatticephi}, there exists a quaternion algebra other than $\textrm{M}_2(\mathds{Q})$ with an order $\mathcal{O}$ such that $\phi_n(\mathcal{O}^1)<\Lambda$ up to conjugation. Since the quaternion algebra is not $\textrm{M}_2(\mathds{Q})$, $\mathcal{O}^1$ is a uniform lattice of $\SL(2,\mathds{R})$. Hence it contains a finite index surface subgroup.

Let $g\geq2$ such that there exists an embedding $j:\pi_1(S_g)\hookrightarrow\mathcal{O}^1$. Then $\phi_n\circ j$ is a diagonal maximal representation with image in $\Lambda$ up to conjugation. In this section, we construct Zariski-dense maximal representations in $\Lambda$ that are continuous deformations of $\phi_n\circ j$.

\subsection{Centralizer of a diagonal element}

The representations that we will build in Section \ref{subsectionproofoftheoremdiagonal} will be bendings of $\phi_n\circ j$ along a simple closed curve. This requires to find such a curve with big enough centralizer in $\Lambda$. In this section, we show that for suitable $\Lambda$, diagonal elements have a big centralizer in $\Lambda$.

Let $\lambda>1$ and $D=\Diag(\lambda,\lambda^{-1})\in\SL(2,\mathds{R})$.
The centralizer of $\phi_n(D)$ in $\textrm{M}_{2n}(\mathds{C})$ is the set of matrices of the form \begin{equation}
    \label{equationcommutator}
        \begin{pmatrix}
            a_{11}&0&a_{12}&0&\dots&a_{1n}&0\\
            0&b_{11}&0&b_{12}&\dots&0&b_{1n}\\
            
            a_{21}&0&a_{22}&0&\dots&a_{2n}&0\\

            0&b_{21}&0&b_{22}&\dots&0&b_{2n}\\
            \dots&\dots&\dots&\dots&\dots&\dots&\dots\\
            a_{n1}&0&a_{n2}&0&\dots&a_{nn}&0\\
            0&b_{n1}&0&b_{n2}&\dots&0&b_{nn}
        \end{pmatrix}
    \end{equation} where $(a_{ij}),(b_{ij})\in\textrm{M}_n(\mathds{C})$. If one interprets $\textrm{M}_{2n}(\mathds{C})$ as $\textrm{M}_n(\textrm{M}_2(\mathds{C}))$, i.e. as $n$-by-$n$ matrices with coefficients in $\textrm{M}_2(\mathds{C})$, $\phi_n(D)$ is a homothety. Matrices of the form \eqref{equationcommutator} are thus $n$-by-$n$ matrices whose coefficients in $\textrm{M}_2(\mathds{C})$ commute with $D$.

\begin{lemma}
\label{lemmacommutatorphiA}
    The centralizer of $\phi_n(D)$ in $\Sp(2n,\mathds{C})$ is the group of matrices of the form \eqref{equationcommutator}
     with $(a_{ij})\in\GL(n,\mathds{C})$ and $(b_{ij})=(a_{ij})^{-\top}$.
\end{lemma}

\begin{proof}
Let $M\in\Sp(2n,\mathds{C})$ that commutes with $\phi_n(D)$. Consider $P$ the matrix defined on the canonical basis of $\mathds{C}^{2n}$ by $Pe_{k}=e_{2k-1}$ and $Pe_{n+k}=e_{2k}$ for all $1\leq k\leq n$. Then
\begin{equation*}
    P^{-1}MP=\begin{pmatrix}
        M_1&\\
        &M_2
    \end{pmatrix}
\end{equation*} with $M_1,M_2\in\textrm{M}_n(\mathds{C})$. The matrix $M$ preserves $K_n$, i.e. $M^{\top}K_nM=K_n$, if and only if $P^{-1}MP$ preserves $P^\top K_nP$. Since
\begin{equation*}
    P^\top K_nP=\begin{pmatrix}
        0&I_n\\
        $-$I_n&0
    \end{pmatrix},
\end{equation*}$M_2=M_1^{-\top}$.
\end{proof}

Let $F$ be a totally real number field and $(a,b)_F$ a quaternion algebra that splits at exactly one real place of $F$. Let $\eta:\Gal(\Qbar/F)\to\PO(n,\Qbar)$ a $1$-cocycle. 

\begin{lemma}
\label{lemmastableconjugation}
    For all $\sigma\in\Gal(\Qbar/F)$, the centralizer of $\phi_n(D)$ in $\Sp(2n,\overline{\mathds{Q}})$ is stable under conjugation by $\eta(\sigma)\otimes T^{a,b}_{\sigma}$.
\end{lemma}

\begin{proof}
    Fix $\sigma\in\Gal(\Qbar/F)$. The centralizer of $\phi_n(D)$ is stable under conjugation by $\eta(\sigma)\otimes I_2$ since the latter lies in it. Moreover, it is stable under conjugation by $I_n\otimes T^{a,b}_{\sigma}$ as shown by an explicit computation computations. Hence it is stable by their product.
\end{proof}

Consider the $F$-algebraic group $\mathds{H}$ that associated to any $F$-algebra $\textrm{R}$ the subgroup of $\Sp(2n,\textrm{R})$ that consists of matrices of the form \eqref{equationcommutator} with $\det((a_{ij})_{ij})=1$. It is isomorphic to the $F$-algebraic group $\SL_n$. The main result of this section is the following lemma.

\begin{lemma}
\label{lemmacommutatorlattice}
    Let $\sigma_0$ be the only real place of $F$ over which $(a,b)_F$ splits. Then $\mathds{H}(\mathds{R}^{\sigma_0})\cap\prescript{}{\eta\otimes T^{a,b}}{\Sp_{2n}}(\mathcal{O}_F)$ is a lattice in $\mathds{H}(\mathds{R}^{\sigma_0})$.
\end{lemma}

\begin{proof}
    For all $\sigma\in\Gal(\Qbar/F)$, $\eta(\sigma)\otimes T^{a,b}_{\sigma}$ induces an automorphism of $\mathds{H}(\Qbar)$ as shown by Lemma \ref{lemmastableconjugation}. Hence $\eta\otimes T^{a,b}$ defines a $1$-cocycle in $\Aut(\mathds{H}(\Qbar))$ and thus correspond to an $\Qbar/F$-form of $\mathds{H}$. Denote the latter by $\mathds{G}$. Since $\mathds{G}$ is semisimple, Borel-Harish-Chandra Theorem \cite{Borel_Arithmeticsubgroupsalgebraicgroups} implies that $\mathds{G}(\mathcal{O}_F)$ is a lattice of $$\prod_{\sigma:F\hookrightarrow\mathds{R}}\mathds{G}(\mathds{R}^{\sigma}).$$
    For all $\sigma\neq\sigma_0$, $\mathds{G}(\mathds{R}^{\sigma})$ is a closed subgroup of $\prescript{}{\eta\otimes T^{a,b}}{\Sp_{2n}(\mathds{R}^{\sigma})}$ which is a compact Lie group.
    Hence $\mathds{G}(\mathcal{O}_F)$ is a lattice of $\mathds{G}(\mathds{R}^{\sigma_0})=\mathds{H}(\mathds{R}^{\sigma_0})$.
\end{proof}

\subsection{Mapping class group orbits}

To prove that the construction given in the next section gives rise to infinitely many mapping class group orbits of representations, we use the Strong Approximation Theorem. The following theorem is an adaptation to our context of the one in \cite{Weisfeiler_StrongapproximationZariskidensesubgroups}. See \S5.2 of \cite{Audibert_ZariskidenseHitchinrepresentationsuniformlattices} for the technical details.

\begin{theorem}[Strong Approximation, Weisfeiler \cite{Weisfeiler_StrongapproximationZariskidensesubgroups}]
\label{propstrongapproximationadapted}
Let $F$ be a totally real number field and $\mathds{G}$ be a connected, almost simple and simply-connected $F$-algebraic group which is compact over all real places of $F$ except one. Let $\Gamma<\mathds{G}(F)$ be a finitely generated Zariski-dense subgroup. There exists a finite index subgroup $\Gamma'<\Gamma$, $a\in\mathcal{O}_F$ and a group scheme structure $\mathds{G}_0$ over $({\mathcal{O}_F})_a$ on $\mathds{G}$ such that $\Gamma'<\mathds{G}_0(({\mathcal{O}_F})_a)$ and for all prime ideals $\mathcal{P}$ of $\mathcal{O}_F$ but finitely many, $\Gamma'$ surjects onto $$\mathds{G}_0\left(\mathcal{O}_F\!\raisebox{-.65ex}{\ensuremath{/\mathcal{P}}}\right).$$
\end{theorem}

Let $F$ be a totally real number field. Let $\mathcal{O}$ be an order of a quaternion algebra $A$ that splits at exactly one real place of $F$. Let $\Gamma<\SU(I_n,\overline{\phantom{s}};\mathcal{O})$ be a finitely generated Zariski-dense subgroup.
Consider the $F$-algebraic group defined as a functor from $F$-algebras to groups by
$$R\mapsto\SU(I_n,\overline{\phantom{s}};A\otimes_FR):=\{M\in\SL_n(A\otimes_FR)\mid \overline{M}^tM=I_n\}$$ where the conjugation on $A$ is extended trivially on $A\otimes_FR$. It is a connected, almost simple and simply-connected algebraic group since it is an $\Qbar/F$-form of $\Sp_{2n}$. Furthermore it is non-compact over only one real place since $A$ splits only at real place. The Strong Approximation Theorem implies that there exists $a\in\mathcal{O}_F$ such that $\Gamma$ surjects $\Sp(2n,\mathcal{O}_F/\mathcal{P})$ for any prime ideal $\mathcal{P}$ not containing $a$. 

Let $\mathcal{P}$ be a prime ideal not containing $a$. Denote by $\mathds{F}_q$ its residue field and by $\pi:\SU(I_n,\overline{\phantom{s}};\mathcal{O})\to\Sp(2n,\mathds{F}_q)$ the reduction map.
It is not true that two Zariski-dense subgroups $\Gamma_1$ and $\Gamma_2$ of $\SU(I_n,\overline{\phantom{s}};\mathcal{O})$ that are conjugate under $\Sp(2n,\mathds{R})$ satisfy $\pi(\Gamma_1)=\pi(\Gamma_2)$. Nevertheless, we have the following.

\begin{lemma}
\label{lemmareduction}
     Let $\Gamma<\SU(I_n,\overline{\phantom{s}};\mathcal{O})$ be a Zariski-dense subgroup and let $g\in\GL(2n,\mathds{C})$ such that $g\Gamma g^{-1}<\SU(I_n,\overline{\phantom{s}};\mathcal{O})$. If $\pi(\Gamma)=\Sp(2n,\mathds{F}_q)$, then $\pi(g\Gamma g^{-1})=\Sp(2n,\mathds{F}_q)$.
\end{lemma}

\begin{proof}
    We first show that there is a $\mathds{C}$-scalar multiple of $g$ that lies in $\GL(n,A)$. Consider the $F$-subalgebra of $\textrm{M}_n(A)$ generated by $\Gamma$, i.e$.$
    $$F\Gamma=\{\textstyle\sum a_i\gamma_i\mid\gamma_i\in\Gamma,\ a_i\in F\}.$$ Proposition 2.2 in \cite{Bass_Groupsofintegralrepresentationtype} shows that $F\Gamma$ is a central simple algebra of dimension $4n^2$. Hence $F\Gamma=\textrm{M}_n(A)$. It follows that the conjugation by $g$ preserves $\textrm{M}_n(A)$. Since all automorphisms of $\textrm{M}_n(A)$ are inner, see Corollary 2.9.9 in \cite{Maclachlan_ArithmeticHyperbolic3Manifolds}, we can assume that $g\in\GL(n,A)$.

    Denote by $F_{\mathcal{P}}$ the completion of $F$ at $\mathcal{P}$ for the valuation $\nu_{\mathcal{P}}$ and $R_{\mathcal{P}}$ its ring of integers. We can view $\Gamma$ and $g$ in $\GL(n,A\otimes_{F}F_{\mathcal{P}})$. Up to scaling, we can assume that $g\in\textrm{M}_n(\mathcal{O}\otimes_{\mathcal{O}_F}R_{\mathcal{P}})$ and that $g$ is non-trivial modulo $\mathcal{P}$. Let $\Nrd$ denote the reduced norm of $\textrm{M}_n(A\otimes_{F}F_{\mathcal{P}})$, i.e$.$ the composition of the embedding in $\textrm{M}_{2n}(\overline{F}_{\mathcal{P}})$ with the determinant, where $\overline{F}_{\mathcal{P}}$ is the algebraic closure of $F_{\mathcal{P}}$. We claim that $\nu_{\mathcal{P}}(\Nrd(g))=0$.
    
    There exists $h\in\textrm{M}_n(\mathcal{O}\otimes_{\mathcal{O}_F}R_{\mathcal{P}})$ such that $gh=\Nrd(g)$. This follows from the fact that there exists a polynomial $\textrm{P}$ with coefficients in $R_{\mathcal{P}}$ with constant term $\Nrd(g)$ such that $\textrm{P}(g)=0$. Suppose that $\nu_{\mathcal{P}}(\Nrd(g))\geq1$. Let $k\in\mathds{N}$ such that $h$ is trivial modulo $\mathcal{P}^k$ but not modulo $\mathcal{P}^{k+1}$.
    Reducing $g\Gamma h\subset\Nrd(g)\textrm{M}_n(\mathcal{O}\otimes_{\mathcal{O}_F}R_{\mathcal{P}})$ modulo $\mathcal{P}$ we have
    $$\overline{g}\ \Sp(2n,\mathds{F}_q)\overline{h}=0.$$
    Since $\overline{g}\neq0$, $\overline{h}=0$. Hence $k\geq1$.
    
    From $\Nrd(g)^{2n-2}gh=\Nrd(h)$, we see that viewed in a splitting field of $A\otimes_FF_{\mathcal{P}}$, $\Nrd(g)^{2n-2}g$ is the cofactor matrix of $h$. Since $h$ is trivial modulo $\mathcal{P}^k$, its cofactor matrix is trivial modulo $\mathcal{P}^{(2n-1)k}$. Hence $\nu_{\mathcal{P}}(\Nrd(g))> k$.
    Let $\omega$ be a uniformiser of $R_{\mathcal{P}}$. Reducing the equation
    $$g\Gamma \frac{h}{\omega^k}\subset\frac{\Nrd(g)}{\omega^k}\textrm{M}_n(\mathcal{O}\otimes_{\mathcal{O}_F}R_{\mathcal{P}})$$ modulo $\mathcal{P}$, we have that $\frac{\overline{h}}{\omega^k}$ is trivial. This is a contradiction.
    
    Hence $\nu_{\mathcal{P}}(\Nrd(g))=0$, i.e$.$ the reduction of $g$ modulo $\mathcal{P}$ is invertible. This implies that $\pi(g\Gamma g^{-1})=\pi(g)\pi(\Gamma) \pi(g)^{-1}$ which concludes the proof.
\end{proof}

\subsection{Proof of Theorem \ref{theoremdiagonal}}
\label{subsectionproofoftheoremdiagonal}

Write the canonical decomposition
$$\mathds{R}^{2n}=\textrm{V}_1\oplus\ldots\oplus\textrm{V}_n$$ where $\textrm{V}_i=\mathds{R}^2$. Denote by $\Delta$ the subgroup of $\Sp(2n,\mathds{R})$ that preserves this decomposition. Explicitely
\begin{equation*}
    \Delta=\left\{\begin{pmatrix}
        A_1&&\\
        &\dots&\\
        &&A_n
    \end{pmatrix}\Big|\ A_i\in\SL(2,\mathds{R}) \right\}.
\end{equation*}

\begin{lemma}
\label{lemmamaximalzariskiclosure}
Let $H$ be the Zariski closure of a maximal representation in $\Sp(2n,\mathds{R})$. Suppose that $\Delta<H$ and that $H$ acts irreducibly on $\mathds{R}^{2n}$. Then $H=\Sp(2n,\mathds{R})$.
\end{lemma}

\begin{proof}
    Corollary 4 in \cite{Burger_TighthomomorphismsHermitiansymmetricspaces} states that $H$ is reductive and its associated symmetric space is Hermitian. Denote by $\mathfrak{h}$ the Lie algebra of $H$ that we see embedded in $\mathfrak{sp}(2n,\mathds{R})$. Since $\Delta<H$, the centralizer of $H$ in $\Sp(2n,\mathds{R})$ has to be discrete. Hence $\mathfrak{h}$ has no center and so is semisimple.

    Write $\mathfrak{h}=\mathfrak{k}\oplus\mathfrak{h}_0$ where $\mathfrak{k}$ is a maximal compact ideal in $\mathfrak{h}$. Consider a $\mathfrak{sl}_2(\mathds{R})$-subalgebra of $\mathfrak{h}$. We denote by $\pi$ the composition of the inclusion with the projection to $\mathfrak{k}$
    \begin{equation*}
    \pi:\mathfrak{sl}_2(\mathds{R})\hookrightarrow\mathfrak{k}\oplus\mathfrak{h}_0\rightarrow\mathfrak{k}.
    \end{equation*}
    The kernel of $\pi$ cannot be trivial since a compact algebra does not contain $\mathfrak{sl}_2(\mathds{R})$. Since $\mathfrak{sl}_2(\mathds{R})$ is simple, $\pi$ has to be trivial. Hence $\mathfrak{sl}_2(\mathds{R})$ is a subalgebra of $\mathfrak{h}_0$. The Lie algebra of $\Delta$ is then also a subalgebra of $\mathfrak{h}_0$. This implies that $\mathfrak{k}$ centralizes the Lie algebra of $\Delta$, so $\mathfrak{k}$ is trivial.
    
    Corollary 4 in \cite{Burger_TighthomomorphismsHermitiansymmetricspaces} shows that the inclusion $\mathfrak{h}\hookrightarrow\mathfrak{sp}(2n,\mathds{R})$ is tight. For Lie algebras without compact factors, tight embeddings are classfied by Hamlet-Pozzetti in \cite{Hamlet_Classificationtighthomomorphisms} \S5.2. In particular, Table 6 and its interpretation show that if $\mathfrak{h}$ were non-simple then it would act reducibly on $\mathds{R}^{2n}$. Knowing that $H$ contains $\Delta$, the classification implies that $H=\Sp(2n,\mathds{R})$.   
\end{proof}

We can now prove Theorem \ref{theoremdiagonal}.

\begin{proof}[Proof of Theorem \ref{theoremdiagonal}]
    Let $n\geq2$. Let $\Lambda$ be a lattice in $\Sp(2n,\mathds{R})$ not widely commensurable with $\Sp(4k+2,\mathds{Z})$ when $n=2k+1$. By Proposition \ref{propclassifactionlattice} and Proposition \ref{propclassificationlatticephi}, there exists a totally real number field $F$ and a quaternion division algebra $A$ over $F$ such that $A$ splits over exactly one real place $\sigma_0$ of $F$ and an order $\mathcal{O}$ of $A$ satisfying $$\phi_n(\mathcal{O}^1)<\Lambda.$$ Since $A$ is a division algebra, $\mathcal{O}^1$ embeds as a cocompact lattice in $\SL(2,\mathds{R})$. Let $g\geq2$ and $j:\pi_1(S_g)\hookrightarrow\mathcal{O}^1$ be an embedding.

    We can write $A=(a,b)_F$ with $a,b\in\mathcal{O}_F$ such that $a$ and $b$ are positive at $\sigma_0$. Hence the norm $1$ elements of $A$ embed in $\SL(2,\mathds{R})$ as
    \begin{equation*}
   \begin{Bmatrix}
   \begin{pmatrix}
   x_0+\sqrt{a}x_1&\sqrt{b}x_2+\sqrt{ab}x_3\\
   \sqrt{b}x_2-\sqrt{ab}x_3&x_0-\sqrt{a}x_1
   \end{pmatrix}\ |\ x_i\in F,\ \Det=1
   \end{Bmatrix}.
\end{equation*}
Lemma 5.1 in \cite{Audibert_Zariskidensesurfacegroups} shows that there exists a simple closed curve $\gamma$ on $S_g$ with image in $\SL(2,\mathds{R})$ which is diagonal.\footnote{The proof of Lemma 5.1 in \cite{Audibert_Zariskidensesurfacegroups} can be adapted from $\mathds{Z}$ to $\mathcal{O}_F$ by noting that we cannot have $0<\sigma(x_0^2-ax_1^2)<1$ for all real embeddings of $F$ in $\mathds{R}$, since the Galois norm of $x_0^2-ax_1^2$ is an integer.} Note that $A^1=\prescript{}{T^{a,b}}{\SL_2(F)}$ and thus $\Lambda=\prescript{}{\eta\otimes T^{a,b}}{\Sp_{2n}(\mathcal{O}_F)}$.

There exist automorphisms $\varphi_i$ of $\pi_1(S_g)$, $1\leq i\leq n$, that fixe $\gamma$ such that the representations $j_i:=j\circ\varphi_i$ are pairwise non $\PGL(2,\mathds{R})$-conjugate. These automorphisms can be induced by Dehn twists along a simple closed curve disjoint from $\gamma$. Define
    $$\rho:\pi_1(S)\xhookrightarrow{\prod j_i}\SL(2,\mathds{R})^n\hookrightarrow\Sp(2n,\mathds{R})$$
where the last map is the diagonal embedding. The image of $\rho$ is in $\Lambda$ and has Zariski-closure $\SL(2,\mathds{R})^n$ by Lemma 5.13 in \cite{Audibert_ZariskidenseHitchinrepresentationsuniformlattices}. Furthermore it is a continuous deformation of $\phi_n\circ j$.

Let $\mathcal{C}$ be the subgroup of $\Sp(2n,\mathds{R})$ that consists of matrices of the form \eqref{equationcommutator} with $\det((a_{ij})_{ij})=1$. It is isomorphic to $\SL(n,\mathds{R})$ and centralize $\rho(\gamma)$.
By Lemma \ref{lemmacommutatorlattice}, $\Lambda\cap\mathcal{C}$ is a lattice in $\mathcal{C}$. Hence by Borel's Density Theorem \cite[Corollary 4.5.6]{Morris_IntroductionArithmeticGroups}, there exists $B\in\Lambda\cap\mathcal{C}$ such that if $I\subset\{1,\dots,n\}$ satsifies $$B(\oplus_{i\in I}\textrm{V}_i)=\oplus_{j\in J}\textrm{V}_j$$ for some $J\subset\{1,\dots,n\}$ then $I=\varnothing$ or $I=\{1,\dots,n\}$. We will deform the representation $\rho$ by $B$. If $\gamma$ is separating, we can write $\pi_1(S_g)$ as an amalgamated product along $\gamma$. If $\gamma$ is non-separating, we can write $\pi_1(S_g)$ as an HNN-extension. In any case, we define $\rho_{B}$ to be the bending of $\rho$ by $B$, as introduced by Johnson-Millson \cite{Johnson_DeformationSpacesCompactHyperbolicManifolds}. Since $B$ can be continuously deformed into $I_n$ within $\mathcal{C}$, $\rho_{B}$ can be continuously deformed into $\rho$. Hence it is maximal.

We show that $\rho_{B}$ acts irreducibly on $\mathds{R}^{2n}$. Let $\textrm{V}\subset\mathds{R}^{2n}$ be an invariant subset of $\rho_{B}$. Since the Zariski-closure of $\rho_{B}$ contains $\Delta$, $\textrm{V}$ has to be invariant under $\Delta$. Hence there exists $I\subset\{1,\dots,n\}$ such that $\textrm{V}=\oplus_{i\in I}\textrm{V}_i$. If $\gamma$ is separating, $B\textrm{V}$ has to be invariant by $\Delta$. If $\gamma$ is non-separating, $\textrm{V}$ has to be invariant by $B$. In any case, we conclude that $\textrm{V}$ is either $\{0\}$ or $\mathds{R}^{2n}$.
Hence $\rho_{B}$ acts irreducibly on $\mathds{R}^{2n}$. Its Zariski-closure contains $\Delta$ so Lemma \ref{lemmamaximalzariskiclosure} implies that $\rho_{B}$ is Zariski-dense in $\Sp(2n,\mathds{R})$.

We claim that the sequence of representations $(\rho_{B^k})_{k\geq1}$ give rise to infinitely many $\MCG(S_g)$-orbits of representations. Suppose not. Denote by $\Gamma_k$ the image of $\rho_{B^k}$. There exist $k_1,\dots,k_l$ such that every $\Gamma_k$ is conjugated to one of the $\Gamma_{k_i}$, $1\leq i\leq l$. By the Strong Approximation Theorem, see Theorem \ref{propstrongapproximationadapted}, there exists a finite set $\Omega$ of prime ideals of $\mathcal{O}_F$ such that for every $1\leq i\leq l$ and every prime ideal $\mathcal{P}\not\in\Omega$, the reduction of $\Gamma_{k_i}$ modulo $\mathcal{P}$ is $\Sp(2n,\mathcal{O}_F/\mathcal{P})$. By Lemma \ref{lemmareduction}, the reduction of every $\Gamma_k$ modulo $\mathcal{P}\not\in\Omega$ is surjective. Pick $\mathcal{P}\not\in\Omega$ and $k$ such that $B^k$ is trivial modulo $\mathcal{P}$. Then the reduction of $\rho_{B^k}$ modulo $\mathcal{P}$ is equal to the reduction of $\rho$ and thus is not surjective. This is a contradiction.
\end{proof}

\noindent
\textbf{Acknowledgement.}
The author is indebted to Andrés Sambarino for his guidance, to Xenia Flamm for her careful reading of this article and to the referee for useful comments.

\bibliographystyle{alpha}
\bibliography{main}

\end{document}